\documentclass[psamsfonts]{amsproc}

\usepackage{mathtools}
\usepackage{dsfont}
\usepackage{upref}
\usepackage{amssymb}
\usepackage[all]{xy}

\newtheorem{theorem}{Theorem}[section]
\newtheorem{lemma}[theorem]{Lemma}
\newtheorem{proposition}[theorem]{Proposition}
\newtheorem{definition-proposition}[theorem]{Definition-Proposition}

\theoremstyle{definition}
\newtheorem{definition}[theorem]{Definition}
\newtheorem{example}[theorem]{Example}

\theoremstyle{remark}
\newtheorem{remark}[theorem]{Remark}

\numberwithin{equation}{section}

\newcommand{\cO}{\mathcal{O}}

\begin{document}

\title{Supergrassmannians as Homogeneous Superspaces}

\author[Mohammadi M.]{Mohammad Mohammadi}
\address{Institute for Advanced Studies in Basic Sciences (IASBS)\\Gavazank, Zanjan, Iran}
\curraddr{}
\email{moh.mohamady@iasbs.ac.ir}
\thanks{}

\author[Varsaie S.]{Saad Varsaie}
\address{
Institute for Advanced Studies in Basic Sciences (IASBS)\\Gavazank, Zanjan, Iran}
\curraddr{}
\email{varsaie@iasbs.ac.ir}
\thanks{}

\date{}

\begin{abstract}
A homogeneous space is a manifold on which a Lie group acts transitively. Super generalization of this concept is also studied in \cite{g-spaces} and \cite{carmelibook}. In this paper we explicitly show that super Lie group $GL(m|n)$ acts transitively on supergrassmannian $G_{k|l}(m|n)$. In this regard, by using functor of point approach, this action is constructed by gluing local actions.
\end{abstract}

\maketitle

\section*{Introduction}
In common geometry, a homogeneous space defined as the quotient space of a Lie group by an its closed subgroup. These spaces are important for pure mathematical studies as well as for their applications to physics. See \cite{zeghib} and \cite{arkani}. It is shown that a manifold is a homogeneous space if a Lie group acts on it transitively \cite{Greub2}.

\medskip

Extension of these spaces in supergeometry is also studied. For example in \cite{carmelibook}, \cite{g-spaces} and \cite{qoutiontsuper}, the concepts of homogeneous superspaces, super Lie groups and their actions on supermanifolds  are introduced by sheaf theoretic approach. In this paper, we show that the super Lie group $GL(m|n)$, c.f. section \ref{prelim}, acts transitively on supergrassmannian $G_{k|l}(m|n)$ , c.f. section \ref{supergrass}. Thus it is a non trivial example of Balduzzi et al's homogeneous superspaces. Regarding homogeneity of $G_{k|l}(m|n)$, there exists other attempts which have been done earlier. But these works limited to study of special cases or the infinitesimal actions of the groups (see \cite{valenzuela}).

\medskip

In the first section, we recall briefly all necessary basic concepts from \cite{carmelibook} and \cite{g-spaces}. For clarifying the concepts or ideas, whenever it is needed, we give some examples. 

\medskip

In section 2, we study supergrassmannians  extensively. Although this concept is introduced by Manin in \cite{ma1}, but here by developing an efficient formalism, we give a precise proof for existence of these supermanifolds.

\medskip

In section 3, by a functor of points approach, an action of super Lie group $GL(m|n)$ on supergrassmannian
$G_{k|l}(m|n)$ is defined by gluing local actions. Finally it is shown that this action is transitive.

\section{Preliminaries} \label{prelim}
In this section we introduce the basic definitions and results concerning category theory, super Lie groups and action of a super Lie group on a supermanifold. For further detail see \cite{carmelibook}.

\medskip

A supermanifold of dimension $p|q$ is a pair $(|M|,\cO_M)$, where $|M|$ is a  second countable and Hausdorff topological space, and $\cO_M$ is a sheaf of $\mathbb Z_2$-graded algebras, locally isomorphic to $C^\infty(\mathbb R^p)\otimes \wedge \mathbb R^q$. A morphism between two supermanifolds $M=(|M|,\cO_M)$ and $N=(|N|,\cO_N)$ is a continuous map $|\psi|:|M|\rightarrow |N|$ together with a sheaf morphism $\psi^*:\cO_N\rightarrow \cO_M$ called pullback.

\medskip

For open subset $U\subset|M|$, by $J_M(U)$ we mean the set of nilpotent elements in $\cO_M(U)$. It is clear that this set is an ideal in $\cO_M(U)$. So one has the quotient sheaf $\frac{\cO_M}{J_M}$, and locally there exists a canonical isomorphism from this sheaf to the sheaf $C^{\infty}$ under $s+J\mapsto \tilde{s}$. Thus $\tilde M:=(|M|,\frac{\cO_M}{J_M})$ is a classical manifold that is called \textit{reduced manifold} associated to $M$. One can associate a reduced map $\widetilde\psi:\widetilde M\rightarrow \widetilde N$ to each morphism $\psi:M\rightarrow N$. By evaluation of $s$ at $x\in U$, denoted by $ev_x(s)$,
we mean $\tilde{s}(x)$. 

\medskip

By locally small category, we mean a category such that the collection of all morphisms between any two of its objects is a set. Let $X$, $Y$ are objects in a category and $\alpha,\beta:X\rightarrow Y$ are morphisms between these objects. An universal pair $(E,\epsilon)$ is called equalizer if the following diagram commutes:
$$E\xrightarrow{\epsilon} X\overset{\alpha}{\underset{\beta}{\rightrightarrows}}Y$$
i.e., $\alpha \circ \epsilon=\beta \circ \epsilon$ and also for each object $T$ and any morphism $\tau:T\rightarrow X$ which satisfy $\alpha \circ \tau=\beta \circ \tau$, there exists unique morphism $\sigma:T\rightarrow E$ such that $\epsilon \circ \sigma = \tau$. If equalizer existed then it is unique up to isomorphism. For example, in the category of sets, which is denoted by $\textbf{SET}$, the equalizer of two morphisms $\alpha,\beta:X\rightarrow Y$ is the set $E=\{x\in X | \alpha(x)=\beta(x)\}$ together with the inclusion map  $\epsilon:E\hookrightarrow X$

\medskip

Let $\mathcal{C}$ be a locally small category and $X$ is an object in $\mathcal{C}$. By $T$-points of $X$, we mean $X(T):=Hom_\mathcal{C}(T,X)$ for any $T\in Obj(\mathcal{C})$. The functor of points of X is a functor which is denoted by X(.) and is defined as follows:

$$\begin{matrix} X(.):\mathcal{C} \rightarrow \textbf{SET}\\
\qquad\quad S \mapsto X(S)\\
\\
X(.): Hom_{\mathcal{C}}(S,T)\rightarrow Hom_{\textbf{SET}}(X(T),X(S))\\
\varphi \mapsto X(\varphi),\end{matrix}$$\\
where $X(\varphi):f\mapsto f\circ \varphi.$ A functor $F:\mathcal{C} \to \textbf{SET}$ is called representable if there exists an object $X$ in $\mathcal{C}$ such that $F$ and $X(.)$ are isomorphic. Then one may say that $F$ is represented by $X$. 
The category of functors from $\mathcal{C}$ to SET is denoted by $[\mathcal{C}, \textbf{SET}]$. It is shown that the category of all representable functors from $\mathcal{C}$ to SET is a subcategory of $[\mathcal{C}, \textbf{SET}]$.

\medskip

Corresponding to each morphism $\psi:X\rightarrow Y$, there exists a natural transformation
$\psi(.)$ from $X(.)$ to $Y(.)$. This transformation correspond the mapping 
$\psi(T):X(T)\rightarrow Y(T)$ with $\xi\mapsto \psi\circ \xi$ for each $T\in Obj(\mathcal{C})$. Now set:
$$\begin{matrix} \mathcal{Y}:\mathcal{C}\rightarrow [\mathcal C,\textbf{SET}]\\ X \mapsto X(.)\\ \psi \mapsto \psi(.). \end{matrix}$$
Obviously, $\mathcal{Y}$ is a covariant  functor and it is called \textit{\textbf{Yoneda embedding}}.
\begin{lemma}\label{Yoneda}
The Yoneda embedding is full and faithful functor, i.e. the map $$Hom_\mathcal{C}(X,Y)\longrightarrow Hom_{[\mathcal{C},\textbf{SET}]}(X(.),Y(.))$$ 
is a bijection for each $X,Y\in Obj(\mathcal{C})$
\end{lemma}
\begin{proof}
see \cite{carmelibook}.
\end{proof}
Thus according to this lemma, $X,Y\in Obj(\mathcal{C})$ are isomorphic if and only if their functor of points are isomorphic. The Yoneda embedding is an equivalence between $\mathcal{C}$ and a subcategory of representable functors in $[\mathcal{C},\textbf{SET}]$, since not all functors are representable.

\begin{remark}\label{charttheorem}
Consider the superdomain $\mathbb R^{p|q}$ and arbitrary supermanifold $T$. Let $f_i\in\cO(T)_0$ and $g_j\in\cO(T)_1$ such that $1\leq i \leq p$, $1\leq j \leq q$ be even and odd elements respectively. By Theorem 4.3.1 in \cite{vsv}, One can determine a unique morphism $\psi:M\rightarrow \mathbb R^{p|q}$, 
by setting $t_i\mapsto f_i$ and $\theta_j\mapsto g_j$ where $(t_i, \theta_j)$ is a  global even and odd coordinate system on $\mathbb R^{p|q}$. Thus $\psi$ may be represented by $(f_i, g_j)$.
\end{remark}

Let $\textbf{SM}$ be the category of supermanifolds and their morphisms. Obviously, $\textbf{SM}$ is a locally small category and  has finite product property. In addition it has a terminal object $\mathbb{R}^{0|0}$, that is the constant sheaf $\mathbb{R}$ on a singleton $\{0\}$.

\medskip

Let $M=(|M|,\cO_M)$ be a supermanifold and $p\in |M|$. There is a map $j_p=(|j_p|,{j_p}^*)$ where:
$$\begin{matrix} |j_p|:\{0\} \rightarrow |M|\qquad\qquad 
& {j_p}^*:\mathcal{O}_M\rightarrow \mathbb{R} \\
& \qquad\qquad\qquad\quad g \mapsto \tilde g(p)=:ev_p(g).\end{matrix}$$
So, for each supermanifold $T$, one can define the morphism
\begin{align}\hat{p}:T\rightarrow \mathbb{R}^{0|0}\xrightarrow{j_p}M\label{phat}\end{align}
as a composition of $j_p$ and the unique morphism $T\rightarrow \mathbb{R}^{0|0}$.

\medskip

By super Lie group, we mean a group-object in the category $\textbf{SM}$. More precisely, it is defined as follows:
\begin{definition}
A super Lie group $G$ is a supermanifold $G$ together the morphisms  $\mu:G\times G \rightarrow G,\quad i:G \rightarrow G,\quad E:\mathbb{R}^{0|0} \rightarrow G$ say multiplication, inverse and unit morphisms respectively, such that the following equations are satisfied
\begin{align*}
\mu \circ (\mu\times 1_G)&=\mu \circ (1_G\times \mu)\\
\mu \circ (1_G\times \hat{e})\circ \bigtriangleup_G&=1_G=\mu \circ (\hat{e}\times 1_G)\circ \bigtriangleup_G\\
\mu \circ (1_G\times i)\circ \bigtriangleup_G&=\hat{e}=\mu \circ (i\times 1_G)\circ \bigtriangleup_G
\end{align*}
where $\mu, i, E$ are multiplication, inverse and unit morphisms respectively, $1_G$ is identity on $G$ and $\hat{e}$ is the morphism according to (\ref{phat}) for element $e\in |G|$ and also $\Delta_G$ be the diagonal map on $G$.
\end{definition}
Note that, there is a Lie group associated with each super Lie group. Indeed, let $G$ is a super Lie group and $\tilde{G}$ is reduced manifold associated to $G$ and $\tilde{\mu}, \tilde{i}, \tilde{E}$ are reduced morphisms associated to $\mu, i, E$ respectively. Since $G \rightarrow \tilde{G}$ is a functor, $(\tilde{G}, \tilde{\mu}, \tilde{i}, \tilde{E})$ is a group-object of the category of differentiable manifolds.

\medskip
\begin{remark}\label{remark2}
Simply, one can show that any super Lie group $G$ induced a group structure over its $T$-points for any arbitrary supermanifold $T$. This means that the functor $T\rightarrow G(T)$ takes values in category of groups. Moreover, for any other supermanifold $S$ and morphism $T\rightarrow S$, the corresponding map $G(T)\rightarrow G(S)$ is a homomorphism of groups. One can also define a super Lie group as a representable functor $T\rightarrow G(T)$ from category $\textbf{SM}$ to category of groups. If such functor represented by a supermanifold $G$, then the maps $\mu,i,E$ are obtained by Yoneda's lemma and the maps $\mu_T:G(T)\times G(T)\rightarrow G(T),\quad i_T:G(T)\rightarrow G(T)$ and $E_T:\mathbb{R}^{0|0}(T)\rightarrow G(T)$.
\end{remark}
\begin{example}
Let $(t, \theta)$ be a global coordinate system on superdomain $\mathbb{R}^{1|1}$. We use the language of functor of points and define multiplication 
$\mu:\mathbb{R}^{1|1}\times \mathbb{R}^{1|1}\rightarrow \mathbb{R}^{1|1}$.
Let $T$ be an arbitrary supermanifold, we define:
\begin{align*}
   \mu_T:\mathbb{R}^{1|1}(T)\times \mathbb{R}^{1|1}(T)& \longrightarrow  \mathbb{R}^{1|1}(T)
    \\
          (f,g),(f^{\prime},g^{\prime}) & \longmapsto (f+f^{\prime},g+g^{\prime})
\end{align*}
where, according to Remark \ref{charttheorem}, there exists a unique morphism from $T$ to $\mathbb{R}^{1|1}$ corresponding to each of $(f,g),(f^{\prime},g^{\prime})$, where $f,f^{\prime}\in \cO(T)_0$ and $g,g^{\prime}\in\cO_1.$ 
$\mathbb{R}^{1|1}(T)$ with $\mu_T$ is a common group. Thus, by Remark \ref{remark2}, $\mathbb{R}^{1|1}$ is a super Lie group.  
\end{example} 
Analogously, one may show that the superdomain $\mathbb{R}^{p|q}$ is a super Lie group. 
\begin{example}
Let $V$ be a finite dimensional super vector space of dimension $m|n$ and Let $\{R_1,\cdots,R_{m+n}\}$ be a  basis of V for which the first $m$ elements. we denote even and odd elements of an its basis by $r_i$. Consider the functor 
\begin{align*}
F: & \textbf{SM} \rightarrow \textbf{Grp}\\
& T\mapsto Aut_{\mathcal O(T)}(\mathcal O(T)\otimes V)\end{align*}
where $F$ maps each seupermanifold $T$ to the group of even $\mathcal O(T)$-module automorphisms of $\mathcal O(T)\otimes V$, and $\textbf{Grp}$ is the category of groups. Consider the supermanifold $\textbf{End}(V)=\Big(End(V_{0})\times End(V_{1}),\mathcal{A}\Big)$ where $\mathcal{A}$ is the structure sheaf $\textbf{End}(V)$. We denote a basis of $\textbf{End}(V)$ by $\{F_{ij}\}$ where $F_{ij}$ is an even linear transformation on V with $R_k\mapsto \delta_{ik}R_j$. If $\{f_{ij}\}$ is the corresponding dual basis, then it may be considered as a global coordinates on $\textbf{End}(V)$. Let $X$ be the open subsupermanifold of $\textbf{End}(V)$ corresponding with the open set:
$$|X|=GL(V_{0})\times GL(V_{1})\subset End(V_{0})\times End(V_{1})$$
Thus, we have
$$X=\Big(GL(V_{0})\times GL(V_{1}),\mathcal{A}|_{GL(V_{0})\times GL(V_{1})}\Big).$$
It can be shown that the functor $F$ may be represented by $X$. For this, one may show $Hom(T,X)\cong Aut_{\mathcal O(T)}(\mathcal O(T)\otimes V)$. First, we have
$$Hom(T,X)=Hom(\mathcal A_X(|X|),\mathcal O(T))$$
 It is known that each $\psi\in Hom(\mathcal A_X(|X|),\mathcal O(T))$ may be uniquely determined by $\{g_{ij}\}$ where $g_{ij}=\psi(f_{ij})$.
 Now set $\bar{\psi}(R-j)=\Sigma g_{ij}R_i$.
 One may consider $\bar{\psi}$ as an element of $Aut_{\mathcal{O}(T)}(\mathcal{O}(T)\otimes V)$. Obviously $\psi\mapsto \bar{\psi}$ is a bijection from $Hom(T,X)$ to $Aut_{\mathcal{O}(T)}(\mathcal{O}(T)\otimes V)$.
Thus the supermanifold $X$ is a super Lie group and denoted it by $GL(V)$ or $GL(m|n)$ if $V=\mathbb{R}^{m|n}$.
Therefore $T$- points of $GL(m|n)$ is the group of invertible even $m|n\times m|n$ supermatrices $\begin{pmatrix}
A & B\\ C & D \end{pmatrix}$ such that $A=(a_{ij}), B=(b_{il}), C=(c_{kj}), D=(d_{kl})$ with $a_{ij}, d_{kl}\in \mathcal O(T)_0,\quad b_{il},c_{kj}\in \mathcal O(T)_1 $ and the multiplication can be written as the matrix product.
\end{example}
Let $x\in|G|$, one can define the left and right translation by $x$ as 
\begin{align}
r_x:=\mu \circ (1_G\times \hat x)\circ \Delta_G,\label{pullefttrans}\\
l_x:=\mu \circ (\hat x\times 1_G)\circ \Delta_G,\label{pulrighttrans}
\end{align}
respectively. One can show that pullbacks of above morphisms are as following
\begin{align}
r_x^*:=(1_{\cO(G)}\otimes ev_x)\circ \mu^*,\label{lefttrans}\\
l_x^*:=(ev_x\otimes 1_{\cO(G)})\circ \mu^*.\label{righttrans}
\end{align}
One may also use the language of functor of points to describe two morphisms \ref{pullefttrans} and \ref{pulrighttrans}.
\begin{definition}
Let $M$ be a supermanifold and let $G$ be a super Lie group with $\mu,i$ and $E$ as its multiplication, inverse and unit morphisms respectively. A morphism $a:M\times G\rightarrow M$ is called a (right) action of $G$ on $M$, if the following diagrams are commuted:
\begin{Small}\begin{displaymath}
\xymatrix{
& M\times G\times G  \ar[rd]^{1_M\times\mu} \ar[ld]_{  a\times 1_G} & \\
M\times G \ar[dr]_a &  & M\times G \ar[ld]^{a}\\
& M & }\quad
\xymatrix{
& M\times G \ar[rd]^{a} & \\
M\ar[ur]^{(1_M\times\hat{e})\circ \Delta_M} \ar[rr]_{1_M} &  & M}
\end{displaymath}\end{Small}
where $\hat{e}, \Delta_M$ are as above. In this case, we say G acts from right on $M$. One can define left action analogously. 
\end{definition}
According to the above diagrams, one has:
$$\begin{matrix} a\circ(1_M\times\mu)=a\circ(a\times 1_G)&\qquad,\qquad & a\circ(1_M\times\hat{e})\circ \Delta_M=1_M. \end{matrix}$$
By Yoneda lemma (Lemma \ref{Yoneda}), one may consider, equivalently, the action of G as a natural transformation: 
$$a(.):M(.)\times G(.)\rightarrow M(.),$$
such that for each supermanifold $T$, the morphism $a_T: M(T)\times G(T)\rightarrow  M(T)$ is an action of group $G(T)$ on the set $M(T)$. This means:
\begin{enumerate}
\item[1.] $(m.g_1).g_2=m.(g_1g_2),\qquad \forall g_1,g_2\in G(T), \forall m\in M(T).$
\item[2.] $m.\hat{e}=m,\qquad \forall  m\in M(T).$
\end{enumerate}
Let  $p\in |M|$, define a map 
$$\begin{matrix} a_p:G\rightarrow M\\ \qquad a_p:=a\circ (\hat{p}\times 1_G)\circ \Delta_G, \end{matrix}$$
where $\hat{p}$ is the morphism (\ref{phat}) for $p\in |M|$. Equivalently, this map may be defined as 
$$\begin{matrix} (a_p)_T:G(T)\rightarrow M(T)\\
\qquad g\longmapsto \hat{p}.g \end{matrix}$$
One may easily show that $a_p$ has constant rank(see Proposition 8.1.5 in \cite{carmelibook}, for more details). Before next definition, we recall
that a morphism between supermanifolds, say $\psi:M\rightarrow N$ is a submersion at $x\in|M|$, if $(d\psi)_x$ is surjective and $\psi$ is called submersion, if this happens at each point. (For more detail about this, One can refer to \cite{vsv}, \cite{carmelibook}). $\psi$ is a surjective submersion, if in addition $\tilde{\psi}$ is surjective.
\begin{definition}
 Let $G$ acts on $M$ with action $a:M\times G \rightarrow M$. We say that $a$ is transitive, if there exist $p\in |M|$ such that $a_p$ is a surjective submersion.  
 \end{definition}
It is shown that, if $a_p$ be a submersion for one $p\in |M|$, then it is a submersion for all point in $|M|$.
The following proposition will be required in the last section.
\begin{proposition}\label{Transitive}
Let $a:M\times G \rightarrow M$ be an action, $a$ is transitive if and only if $(a_p)_{\mathbb{R}^{0|q}}:G(\mathbb{R}^{0|q})\rightarrow M(\mathbb{R}^{0|q})$ is surjective, where q is the odd dimension of $G$
\end{proposition}
\begin{proof} see proof proposition 9.1.4 in \cite{carmelibook}. \end{proof}
\begin{definition}
Let $G$ be a super Lie group and let $a$ be an action on supermanifold $M$. By stabilizer of $p\in |M|$ we mean a supermanifold $G_p$ equalizing the diagram $$G\overset{a_p}{\underset{\hat{p}} {\rightrightarrows}}M.$$
 \end{definition}
It is not clear that such an equalizer exists. In this regard, there are two propositions
\begin{proposition}\label{Isotropic}
Let $a:M\times G \rightarrow M$ be an action, then
\begin{enumerate} 
\item[1.] The  following diagram admits an equalizer $G_p$ 
$$G\overset{a_p}{\underset{\hat{p}} {\rightrightarrows}}M.$$
\item[2.] $G_p$ is a sub super Lie group of $G$.
\item[3.] the functor $T\rightarrow (G(T))_{\hat{p}}$ is represented by $G_p$, where $(G(T))_{\hat{p}}$ is the stabilizer of $\hat{p}$ of the action of $G(T)$ on $M(T)$.
\end{enumerate} \end{proposition}
\begin{proof} See proof proposition 8.4.7 in \cite{carmelibook}. \end{proof}
 \begin{proposition}\label{equivariant}
 Suppose $G$ acts transitively on $M$. There exists a $G$-equivariant isomorphism 
 \begin{displaymath}\xymatrix{ \dfrac{G}{G_p}\ar[r]^{\cong} & M. } \end{displaymath} \end{proposition} 
 \begin{proof} See proof proposition 6.5 in \cite{g-spaces}. \end{proof}
\section{Supergrassmannian}\label{supergrass}
Supergrassmannians are introduced by Manin in \cite{ma1}. In this section, we study their sheaf structures in more details. For convenience from now we set $\alpha :=k(m-k)+l(n-l)$ and 
$\beta:= l(m-k)+k(n-l)$ and also demonstrate any supermatrix with four blocks, say $B_1, B_2, B_3, B_4$. Upper left and lower right blocks, $B_1, B_4$ are called even blocks and lower left and upper right blocks, $ B_2, B_3 $ are called odd blocks.
By a supergrassmannian,  $G_{k|l}(m|n)$, we mean a
supermanifold which is constructed by gluing superdomains
$\mathbb{R}^{\alpha|\beta}=(\mathbb{R}^\alpha, \, C^{\infty}(\mathbb{R}^\alpha)\otimes \wedge \mathbb{R}^\beta)$ 
as follows:\\
Let $I\subset \{1, \cdots, m \}$ and $ R \subset \{ m+1, \cdots, m+n\}$ be sets with $ k $ and $ l $ elements respectively. The elements of $ I $ are called even indices and the elements of $ R $ are called odd indices. In this case $I|R$ is called  $k|l$-index. Set $U_{I|R}=(|U_{I|R}|, \mathcal{O}_{I|R})$, where
$$|U_{I|R}|=\mathbb{R}^\alpha \quad,\quad \mathcal{O}_{I|R}(\mathbb{R}^\alpha)=\, C^{\infty}(\mathbb{R}^\alpha)\otimes \wedge \mathbb{R}^\beta.$$
Let each superdomain $U_{I|R}$ is labeled by an even $k|l\times m|n$ supermatrix, say $A_{I|R}$ with four blocks, $ B_1, B_{2}, B_3, B_4 $ as above. Except for columns with indices in $I \cup R$, which together form a minor denoted by $M_{I|R}(A_{I|R})$, even and odd blocks are filled from up to down and left to right by $x_a^I, e_b^I$, the even and odd free generators of $\mathcal{O}_{I|R}(\mathbb{R}^\alpha)$, respectively. This process impose an ordering on the set of generators. In addition $ M_{I|R}A_{I|R} $ is supposed to be a unit supermatrix.

For example, let $I=\{1\}, R=\{1,2\}$ and let $ I|R $ be an $1|2$-index in $G_{1|2}(3|3)$. In this case the set of generators of $\mathcal{O}_{I|R}(\mathbb{R}^\alpha)$ is 
\begin{equation*} \{x_1, x_2, x_3, x_4; e_1, e_2, e_3, e_4, e_5\} \end{equation*}
 and $ A_{I|R} $ is:
\begin{equation*}
\left[
\begin{array}{ccc|ccc}
1 & x_1 & x_2 & 0 & 0 & e_5 \\
\hline
0 & e_1 & e_3 & 1 & 0 & x_3 \\
0 & e_2 & e_4 & 0 & 1 & x_4 \\
\end{array}
\right]
\end{equation*}
Note that, in this example,
$\{x_1,e_1,e_2,x_2,e_3,e_4,e_5,x_3,x_4\}$
is corresponding total ordered set of generators. The transition map between two superdomains, $U_{I|R}$ and $U_{J|S}$ is denoted by 
$$g_{I|R,J|S}:\bigg(|U_{I|R}|\cap |U_{J|S}|,\mathcal{O}_{I|R}|_{|U_{I|R}|\cap |U_{J|S}|}\bigg)\rightarrow \bigg(|U_{I|R}|\cap |U_{J|S}|,\mathcal{O}_{J|S}|_{|U_{I|R}|\cap |U_{J|S}|}\bigg).$$
It may be defined whenever $M_{J|S}A_{I|R}$ is invertible. In this case the following equation defines how to change coordinates(transition map)
\begin{equation}\label{transitionmap}
D_{J|S}\bigg(\big(M_{J|S}(A_{I|R})\big)^{-1}A_{I|R}\bigg)=D_{J|S}(A_{J|S}),
\end{equation}
where $D_{J|S}(A_{I|R})$ is a matrix which is remained after omitting $M_{J|S}(A_{I|R})$. Clearly, this map is defined whenever $M_{J|S}(A_{I|R})$ is invertible.
For example in $G_{1|2}(3|3)$ suppose $I=\{2\}, R=\{1,3\}, J=\{3\}, S=\{2,3\} $ , so $I|R, J|S$ are $ 1|2 $-indices. We have:
\begin{equation*}
A_{I|R}=\left[
\begin{array}{ccc|ccc}
x_1 & 1 & x_2 & 0 & e_5 & 0 \\
\hline
e_1 & 0 & e_3 & 1 & x_3 & 0 \\
e_2 & 0 & e_4 & 0 & x_4 & 1 \\
\end{array}
\right], \quad 
A_{J|S}=\left[
\begin{array}{ccc|ccc}
x_1 & x_2 & 1 & e_5 & 0 & 0 \\
\hline
e_1 & e_3 & 0 & x_3 & 1 & 0 \\
e_2 & e_4 & 0 & x_4 & 0 & 1 \\
\end{array}\right],
\end{equation*}
\begin{equation*}
M_{J|S}A_{I|R}=\left[
\begin{array}{c|cc}
 x_2 & e_5 & 0 \\
\hline
 e_3 & x_3 & 0 \\
 e_4 & x_4 & 1 \\
\end{array}
\right].
\end{equation*}
The transition map between $U_{I|R}$ and $U_{J|S}$ is obtained by substituting the above supermatrices in (\ref{transitionmap}).
By a straightforward proof, it can be shown that the following proposition exist:
\begin{proposition} Let $g_{I|R,J|S}$ be as above, then
\begin{enumerate}
\item[1.] $g_{I|R,I|R}=id.$
\item[2.] $g_{I|R,J|S} g_{J|S,I|R}=id.$
\item[3.] $g_{I|R,J|S} g_{J|S,K|T} g_{K|T,I|R}=id.$
\end{enumerate}
\end{proposition}
\begin{proof}
For first equation, note that the map $ g_{I|R,I|R} $ is obtained from the following equality:
\begin{equation*}
D_{I|R}\bigg((M_{I|R}A_{I|R})^{-1}A_{I|R}\bigg)=D_{I|R}A_{I|R}.
\end{equation*}
Where the matrix $ M_{I|R}A_{I|R} $ is identity. So $g_{I|R,J|S}$ is defined by the following equality:
\begin{equation*}
D_{I|R}A_{I|R}=D_{I|R}A_{I|R}.
\end{equation*}
This shows the first equation. For second equality, let $ J|S $ be an another $ k|l $-index, so $ g_{I|R,J|S} $ is obtained by the following equality:
\begin{equation*}
D_{J|S}\bigg((M_{J|S}A_{I|R})^{-1}A_{I|R}\bigg)=D_{J|S}A_{J|S}
\end{equation*}
One may see that $ g_{J|S,I|R}\circ g_{I|R,J|S} $ is obtained by following equality:
\begin{equation*}
D_{I|R}\bigg(\bigg(M_{I|R}\Big((M_{J|S}A_{I|R})^{-1}A_{I|R}\Big)\bigg)^{-1}(M_{J|S}A_{I|R})^{-1}A_{I|R}\bigg)=D_{I|R}A_{I|R}.
\end{equation*}
For left side, we have
 \begin{align*}&=D_{I|R}\Bigg(\bigg((M_{J|S}A_{I|R})^{-1}M_{I|R}A_{I|R}\bigg)^{-1}(M_{J|S}A_{I|R})^{-1}A_{I|R}\Bigg)\\
 &=D_{I|R}\bigg(\bigg((M_{J|S}A_{I|R})^{-1}\bigg)^{-1}(M_{J|S}A_{I|R})^{-1}A_{I|R}\bigg)\\
 &=D_{I|R}\bigg((M_{J|S}A_{I|R})(M_{J|S}A_{I|R})^{-1}A_{I|R}\bigg)=D_{I|R}(A_{I|R})\end{align*}
Accordingly the map $ g_{J|S,I|R}\circ g_{I|R,J|S} $ is obtained by $ D_{I|R}A_{I|R}=D_{I|R}A_{I|R} $ and it shows that this map is identity.
For third equality, it is sufficient to show that the map $ g_{I|R,J|S}og_{J|S,T|P} $ is obtained from
\begin{equation*}
 D_{I|R}((M_{I|R}A_{T|P})^{-1}A_{T|P})=D_{I|R}A_{I|R} .
\end{equation*}
This case obtain from case 2 Analogously. 
\end{proof}
So the sheaves $(U_{I|R}, \mathcal O_{I|R})$ may be glued through the $g_{I|R, J|S}$ to construct the supergrassmannian $G_{k|l}(m|n)$. Indeed, according to \cite{vsv}, the conditions of the above proposition are necessary and sufficient for gluing.
\section {Supergrassmannian as homogeneous superspace}
In this section, regarding homogeneous superspace, we first quickly recall some definitions from \cite{g-spaces} and \cite{carmelibook}. For more information one may see these papers. Let $G=(|G|,\mathcal{O}_G)$ be a super Lie group and  $H=(|H|,\mathcal{O}_H)$ a closed sub super Lie group of $G$.
One can define a supermanifold structure on the topological space $|X|=\dfrac{|G|}{|H|}$ as follows:\\
Let $\mathfrak{g}=Lie(G)$ and $\mathfrak{h}=Lie(H)$ be super Lie algebras corresponding with $G$, $H$. For each $Z\in \mathfrak{g}$, let $D_{Z}$ be the left invariant vector field on $G$ associated with $Z$. For subalgebra $\mathfrak{h}$ of $\mathfrak{g}$ set:
$$\forall U\subset |G| \qquad \mathcal{O}_{\mathfrak{h}}(U):=\{f\in \mathcal{O}_{G}(U)| D_{Z}f=0 \quad on\hspace{5pt} U, \quad \forall Z\in \mathfrak{h}\}.$$
On the other hand, for any open subset $U\subset |G|$ set:
$$\mathcal{O}_{inv}(U):=\{f\in\mathcal{O}_{G}(U)|\quad \forall x_0 \in |H|,\hspace{5pt} r^*_{x_0}f=f \}.$$
If $|H|$ is connected, then $\mathcal{O}_{inv}(U)= \mathcal{O}_{\mathfrak{h}}(U)$.
For each open subset $W\subset |X|=\dfrac{|G|}{|H|}$, the structure sheaf $\mathcal{O}_X$ is defined as following
$$\mathcal{O}_{X}(W):=\mathcal{O}_{inv}(U)\cap \mathcal{O}_{\mathfrak{h}}(U),$$ 
where $U=\pi^{-1}(W).$
One can show that $\mathcal{O}_{X}$ is a sheaf on $|X|$ and the ringed space $X=(|X|,\mathcal{O}_{X})$ is a superdomain locally (For more details ref. to \cite{g-spaces}). So $X$ is a supermanifold and is called homogeneous superspace.
In this section, we want to show that the supergrassmannian $G_{k|l}(m|n)$ is a homogeneous superspace. According to the section 1, it is enough to find a super Lie group which acts on $G_{k|l}(m|n)$ transitively. We show that the super Lie group $GL(m|n)$ acts on supergrassmannian $G_{k|l}(m|n)$ transitively. First, we have to define a morphism 
$a:G_{k|l}(m|n)\times GL(m|n) \rightarrow G_{k|l}(m|n)$. For this, by Yoneda lemma, it is sufficient, for each supermanifold $T$, to define $a_T$: $$a_T:G_{k|l}(m|n)(T)\times GL(m|n)(T) \rightarrow G_{k|l}(m|n)(T).$$
or equivalently define 
$$(a_T)^P: G_{k|l}(m|n)(T) \rightarrow G_{k|l}(m|n)(T).$$
where $P$ is a fixed arbitrary element in $GL(m|n)(T)$. For brevity, we denote $(a_T)^P$ by $\textbf{A}$.
One may consider $GL(m|n)(T)$, as invertible $m|n\times m|n$ super matrix with entries in $\mathcal{O}(T)$, but there is not such a description for $G_{k|l}(m|n)(T)$, because it is not a superdomain. 
We know each supergrassmaanian is constructed by gluing superdomains (c.f. section 2), so one may define the actions of $GL(m|n)$ on superdomains $(|U_{I|R}|,\mathcal{O}(U_{I|R}))$ and then shows that these actions glued to construct $a_T$. 
In addition for defining $\textbf{A}$, it is needed to refine the covering $\{U_{I|R}(T)\}_{I|R}$. Set
$$U_{I|R}^{J|S}(T):=\Big\{X\in U_{I|R}(T)\quad|\quad D_{J|S}\big(\big(M_{J|S}(\tilde{X}_{I|R}P)\big)^{-1}\tilde{X}_{I|R}P\big)\in U_{J|S}(T)\Big\},$$
where, by $\tilde{X}_{I|R}$, we mean a supermatrix that except for columns with indices in $I\cup R$, which together form an identity supermatrix, even and odd blocks are filled from up to down and left to right by $f_a^{I|R}, g_b^{I|R}$, the even and odd free elements of $\cO_{I|R}(T)$. One can show that $\{U_{I|R}^{J|S}(T)\}_{I|R,J|S}$ is a covering for $G_{k|l}(m|n)(T).$ 
Now consider all maps 
\begin{align*} 
\textbf{A}_{I|R}^{J|S}: & U_{I|R}^{J|S}(T)\rightarrow U_{J|S}(T)\\
& \quad X \rightarrow  D_{J|S}\Big(\big(M_{J|S}(\tilde{X}_{I|R}P)\big)^{-1}\tilde{X}_{I|R}P\Big),
\end{align*}
where, $\tilde{X}_{I|R}$ is as above. We have to show that these maps may be glued to construct a global map on $G_{k|l}(m|n)(T)$. For this, it is sufficient to show that
the following diagram commutes:
\begin{Small} \begin{displaymath} \xymatrix{ & U_{I|R}^{J|S}(T)\cap U_{K|Q}^{H|L}(T) \ar[rd]^{(g_{{K|Q},{I|R}})_T}\ar[ld]_{\textbf{A}_{I|R}^{J|S}} & \\ U_{J|S}(T)\cap U_{H|L}(T)\ar[dr]_{(g_{{H|L},{J|S}})_T} &  & U_{I|R}^{J|S}(T)\cap U_{K|Q}^{H|L}(T) \ar[ld]^{\textbf{A}_{K|Q}^{H|L}}\\
& U_{J|S}(T)\cap U_{H|L}(T) & }\quad 
\end{displaymath} \end{Small}
We need the following lemma to show commutativity of the above diagram. Let for each supermanifold $T$, $\big(g_{I|R, J|S}\big)_T$
be the induced map from $g_{I|R, J|S}$ on $T$-points. 
\begin{lemma}
Let $\psi:T\rightarrow \mathbb{R}^{\alpha|\beta}$ be a $T$-point of $\mathbb{R}^{\alpha|\beta}$ and $(z_{tu})$ be a global coordinates of  $\mathbb{R}^{\alpha|\beta}$ with ordering as the one introduced in the second  paragraph at the start of section 2. If $B=(\psi^*(z_{tu}))$ is the supermatrix corresponding to $\psi$, then the supermatrix corresponding to $\big(g_{I|R,J|S}\big)_T(\psi)$ is as follows:
\begin{equation*}
 D_{I|R}((M_{I|R}\tilde{B}_{J|S})^{-1}\tilde{B}_{J|S})
\end{equation*}
where $\tilde{B}_{J|S}$ is as above. (c.f. the second paragraph before this lemma).
\end{lemma}
\begin{proof}
Note that $g^*_{I|R, J|S}$ may be represented by a supermatrix as follows:
\begin{equation*}
 D_{I|R}((M_{I|R}A_{J|S})^{-1}A_{J|S})
\end{equation*}
Let $M_{I|R}A_{J|S}=(m_{tu})$ and $((M_{I|R}A_{J|S})^{-1}=(m^{tu})$. If $z=(z_{ij})$ be a coordinate system including even and odd coordinates on $U_{I|R}$, then one has
$$g^*_{I|R, J|S}(z_{tu})=\sum m^{tk}(z).z_{ku}$$
Then
\begin{align*}
\psi^*\circ g^*_{I|R, J|S}(z_{tu})=\\
\psi^*\big(\sum m^{tk}(z).z_{ku}\big)=\\
\sum m^{tk}(\psi^*(z)).\psi^*(z_{ku})\\
\end{align*}
For second equality one may note that $\psi^*$ is homomorphism of $Z_2$-graded algebras and $m^{tk}(z)$ is rational function of z. 
Obviously, the last expression is the (t, u)-entry of the matrix $D((M_{I|R}B)^{-1}B)$. This completes the proof.
\end{proof}
\begin{proposition}
The above diagram commutes.
\end{proposition}
\begin{proof}
Equivalently, for arbitrary $k|l$-indices $I|R, J|S, K|Q, H|L$ we have to show that
\begin{equation}\label{glueaction}
(g_{{H|L},{J|S}})_T\circ \textbf{A}_{I|R}^{J|S}=\textbf{A}_{K|Q}^
{H|L} \circ (g_{{K|Q},{I|R}})_T.
\end{equation}
Let $X\in U_{I|R}^{J|S}(T)\cap U_{K|Q}^{H|L}(T)$ be an arbitrary element. One has $X\in U_{I|R}^{J|S}(T)$, so 
\begin{align*} D_{J|S}\Big(\big(M_{J|S}(\tilde{X}_{I|R}P)\big)^{-1}\tilde{X}_{I|R}P\Big)&\in U_{J|S}(T),\\
(g_{{H|L},{J|S}})_T\Bigg(D_{J|S}\Big(\big(M_{J|S}(\tilde{X}_{I|R}P)\big)^{-1}\tilde{X}_{I|R}P\Big)\Bigg)&\in U_{H|L}(T).
\end{align*}
From left side of (\ref{glueaction}), we have:
\begin{align*}
(g_{{H|L},{J|S}})_T&\circ \textbf{A}_{I|R}^{J|S}(X)\\
&=(g_{{H|L},{J|S}})_T\Bigg(D_{J|S}\Big(\big(M_{J|S}(\tilde{X}_{I|R}P)\big)^{-1}\tilde{X}_{I|R}P\Big)\Bigg)\\
& =D_{H|L}\Bigg(\bigg(M_{H|L}\Big((M_{J|S}(\tilde{X}_{I|R}P))^{-1}\tilde{X}_{I|R}P\Big)\bigg)^{-1}(M_{J|S}(\tilde{X}_{I|R}P))^{-1}\tilde{X}_{I|R}P\Bigg)\\
& =D_{H|L}\Bigg(\Big((M_{J|S}(\tilde{X}_{I|R}P))^{-1}(M_{H|L}(\tilde{X}_{I|R}P))\Big)^{-1}(M_{J|S}(\tilde{X}_{I|R}P))^{-1}\tilde{X}_{I|R}P\Bigg)\\
&
=D_{H|L}\Bigg((M_{H|L}(\tilde{X}_{I|R}P))^{-1}M_{J|S}(\tilde{X}_{I|R}P)(M_{J|S}(\tilde{X}_{I|R}P))^{-1}\tilde{X}_{I|R}P\Bigg)\\
&
=D_{H|L}\Bigg((M_{H|L}(\tilde{X}_{I|R}P))^{-1}\tilde{X}_{I|R}P\Bigg).
\end{align*}‎
Also, since $X\in U_{I|R}(T)$, so
\begin{equation*}
(g_{{K|Q},{I|R}})_T(X)\in U_{K|Q}(T)\quad ,\qquad \textbf{A}_{K|Q}^{H|L}\bigg((g_{{K|Q},{I|R}})_T(X)\bigg)\in U_{H|L}(T)
\end{equation*}
Then from right side of equation (\ref{glueaction}), we have
\begin{align*}
\textbf{A}_{K|Q}^{H|L}&\circ(g_{{K|Q},{I|R}})_T(X)\\
&=\textbf{A}_{K|Q}^{H|L}\Bigg(D_{K|Q}\bigg((M_{K|Q}\tilde{X}_{I|R})^{-1}\tilde{X}_{I|R}\bigg)\Bigg)\\
&  =D_{H|L}\Bigg(\Big[M_{H|L}\bigg((M_{K|Q}\tilde{X}_{I|R})^{-1}\tilde{X}_{I|R}P\bigg)\Big]^{-1}(M_{K|Q}\tilde{X}_{I|R})^{-1}\tilde{X}_{I|R}P\Bigg)\\
&
=D_{H|L}\Bigg(\Big[(M_{K|Q}\tilde{X}_{I|R})^{-1}M_{H|L}(\tilde{X}_{I|R}P)\Big]^{-1}(M_{K|Q}\tilde{X}_{I|R})^{-1}\tilde{X}_{I|R}P\Bigg)\\
&
=D_{H|L}\Bigg(\Big(M_{H|L}(\tilde{X}_{I|R}P)\Big)^{-1}(M_{K|Q}\tilde{X}_{I|R})(M_{K|Q}\tilde{X}_{I|R})^{-1}\tilde{X}_{I|R}P\Bigg)\\
&
=D_{H|L}\Bigg(\Big(M_{H|L}(\tilde{X}_{I|R}P)\Big)^{-1}\tilde{X}_{I|R}P\Bigg).
\end{align*}
This shows that the above diagram commutes.
\end{proof}
Therefore $GL(m|n)$ acts on $G_{k|l}(m|n)$ with action $a$.
Now it is needed to show that this action is transitive. 
\begin{theorem}
$GL(m|n)$ acts on $G_{k|l}(m|n)$ with above action, transitively.
\end{theorem}
\begin{proof}
By proposition \ref{Transitive}, it is sufficient to show that the map $$(a_X)_{\mathbb{R}^{0|q}}:GL(m|n)(\mathbb{R}^{0|q}) \rightarrow G_{k|l}(m|n)(\mathbb{R}^{0|q})$$ is surjective,
where $q=2mn$ is odd dimension of $GL(m|n)$ and $$X=(X_1,X_2)\in|U_{I|R}|= U_I\times U_R$$ is an element and $X_1$ is the matrix of coordinates of a $k$-dimensional subspace, say $A_1$ in $U_I$, and also $X_2$ is the matrix of coordinates of a $l$-dimensional subspace, say $A_2$ in $U_R$. As an element of $G_{k|l}(m|n)(T)$, one may use (\ref{phat}) and represent $\hat{X}$, as: 
$$\hat{X}=\begin{pmatrix}\bar{X}_1 & 0\\0 & \bar{X}_2\end{pmatrix}$$
where $T$ is an arbitrary supermanifold and $\bar{X}_1$ and $\bar{X}_2$ are unique representations of $A_1$ and $A_2$ in $U_I$ and $U_R$ respectively.
For surjectivity, let $$W=\begin{pmatrix}
A & B\\ C & D \end{pmatrix}\in U_{J|S}(\mathbb{R}^{0|q}),$$
 be an arbitrary element.
We have to show that there exists an element $V\in GL(m|n)(\mathbb{R}^{0|q})$ such that $\hat{X}V=W$.
Since $\bar{X}_1\in G_k(m)$ and $A$ is a matrix with rank $k$ and also given that the Lie group $GL(m)$ acts on manifold $G_k(m)$ transitively, then there exists an invertible matrix $P_{m\times m}\in GL(m)$ such that $\bar{X}_1P=A$. Similarly one may see that there exists an invertible matrix $Q_{n\times n}\in GL(n)$ that $\bar{X}_2Q=D$. In addition, the equations $\bar{X}_1Z=B$ and $\bar{X}_2Z^\prime=C$ have infinite solutions. Choose arbitrary solutions say $H_{m\times n}$ and $N_{n\times m}$ for these equations. Clearly, One can see $V=\left[\begin{array}{c|c}P_{m\times m} & H_{m\times n}\\
\hline N_{n\times m} & Q_{n\times n} \end{array}
\right]_{m|n\times m|n}$ satisfy in  the equation $\hat{X}V=W$. So $(A_X)_{\mathbb{R}^{0|q}}$ is surjective.
By Proposition \ref{Transitive}, this shows that the action of $GL(m|n)$ on $G_{k|l}(m|n)$ is transitive.
\end{proof}
Then according to Proposition \ref{equivariant}, $G_{k|l}(m|n)$ is a homogeneous superspace.
\section{Conclusion}
In this paper, we show that $GL(m|n)$ acts transitively on $G_{k|l}(m|n)$. Thus this supermanifold is a homogeneous superspace. This paper is part of our work on $\nu$-grassmannians, a novel generalization of grassmannian in supergeometry, and super Lie groups which act on them. The most important feature of these spaces is that they carry an odd involution on their sheaf structures. In forthcoming papers, the second author (with his PhD students) show that by using $\nu$-grassmannians, one may generalize the theorem of homotopy classification of vector bundles and the concept of universal Chern classes in supergeometry. To see these results in special cases for $\nu$-projective spaces refer to \cite{vclass} and \cite{vprojective}.

\newpage
\providecommand{\bysame}{\leavevmode\hbox to3em{\hrulefill}\thinspace}
\providecommand{\MR}{\relax\ifhmode\unskip\space\fi MR }

\providecommand{\href}[2]{#2}

\end{document}